\documentclass[12pt,reqno]{amsart}
\usepackage{amsmath,amsthm,amsfonts,amssymb,amscd,  multicol, verbatim, enumerate}
\usepackage{verbatim}
\usepackage{fullpage}

\input{xy}
\usepackage{xcolor}
\usepackage[english]{babel}
\usepackage{cite}
\usepackage{amsfonts}
\usepackage{latexsym}
\usepackage{amsthm}
\usepackage{amsopn}
\usepackage{amsmath}
\usepackage[all]{xy}
\usepackage{verbatim}
\usepackage{amssymb}
\usepackage{mathrsfs}
\usepackage{float}
\usepackage{tikz-cd}
\usepackage{tikz}
\usepackage{fullpage}
\usepackage{amscd}
\usepackage{amscd}

\usepackage{graphicx}

\newcommand{\Hom}{\mathrm{Hom}}

\newtheorem{thm}{Theorem}[section]
\newtheorem*{Theorem*}{Theorem}
\newtheorem*{Corollary*}{Corollary}

\newtheorem{lem}[thm]{Lemma}
\newtheorem{prop}[thm]{Proposition}
\newtheorem{cor}[thm]{Corollary}

\theoremstyle{definition}
\newtheorem{defn}[thm]{Definition}

\theoremstyle{remark}
\newtheorem{rmk}[thm]{Remark}
\newtheorem*{rmk*}{Remark}

\title{Topological semiinfinite tensor (super)modules}
\author[Esposito]{Francesco Esposito}
\address{Francesco Esposito, Dipartimento di Matematica, Universit\`a degli Studi di Padova,
via Trieste 63, 35121 Padova, Italy}
\email{esposito@math.unipd.it}
\author[I. Penkov]{Ivan Penkov}
\address{Ivan Penkov, Constructor University, 28759 Bremen, Germany}
\email{ipenkov@constructor.university}

\begin{document}
	
	\maketitle
	
	\begin{abstract}
		We construct universal monoidal categories of topological tensor supermodules over the Lie superalgebras $\mathfrak{gl}(V\oplus \Pi V)$ and $\mathfrak{osp}(V\oplus \Pi V)$ associated with a Tate space $V$.
        Here $V\oplus \Pi V$ is a $\mathbb{Z}/2\mathbb{Z}$-graded topological vector space whose even and odd parts are isomorphic to $V$.
        We discuss the purely even case first, by introducing monoidal categories 
        $\widehat{\mathbf{T}}_{\mathfrak{gl}(V)}$, $\widehat{\mathbf{T}}_{\mathfrak{o}(V)}$ and 
        $\widehat{\mathbf{T}}_{\mathfrak{sp}(V)}$, 
        and show that these categories are anti-equivalent to respective previously studied categories
	$\mathbb{T}_{\mathfrak{gl}(V)}$, $\mathbb{T}_{\mathfrak{o}(V)}$, $\mathbb{T}_{\mathfrak{sp}(V)}$. 
        These latter categories have certain universality properties as monoidal categories, which consequently carry over to
        $\widehat{\mathbf{T}}_{\mathfrak{gl}(V)}$, $\widehat{\mathbf{T}}_{\mathfrak{o}(V)}$ and 
        $\widehat{\mathbf{T}}_{\mathfrak{sp}(V)}$. Moreover,
        the categories $\mathbb{T}_{\mathfrak{o}(V)}$ and $\mathbb{T}_{\mathfrak{sp}(V)}$ are known to be equivalent, and this implies the equivalence of the categories 
        $\widehat{\mathbf{T}}_{\mathfrak{o}(V)}$ and 
        $\widehat{\mathbf{T}}_{\mathfrak{sp}(V)}$.
        After introducing a supersymmetric setting, we establish the equivalence of the category 
        $\widehat{\mathbf{T}}_{\mathfrak{gl}(V)}$ with the category 
        $\widehat{\mathbf{T}}_{\mathfrak{gl}(V\oplus \Pi V)}$, and the equivalence of both categories $\widehat{\mathbf{T}}_{\mathfrak{o}(V)}$ and 
        $\widehat{\mathbf{T}}_{\mathfrak{sp}(V)}$ with 
        $\widehat{\mathbf{T}}_{\mathfrak{osp}(V\oplus \Pi V)}$.
		
		
	\end{abstract}
	
	\section{Introduction}

Let $V$ be a complex infinite-dimensional self-dual Tate space. We think of $V$ as $U\oplus U^\ast$ for a countable-dimensional discrete complex vector space $U$ and its dual $U^\ast$, and $V$ is endowed with a locally linearly compact topology.
The Lie algebra $\mathfrak{gl}(V)$ of continuous endomorphisms of the topological vector space $V$ has been considered by Tate in \cite{T}, where he gave a definition of residues on a curve and a proof of the residue theorem in terms of a central extension of $\mathfrak{gl}(V)$. This approach is made fully explicit in \cite{ADCK}, where the authors reinterpret Tate's work and prove the law of reciprocity for curves in terms of a corresponding central extension of the infinite-dimensional group $\mathrm{GL}(V)$.
The Lie algebra $\mathfrak{gl}(V)$ and its central extension have been studied also as symmetries of nonlinear integrable systems (e.g. \cite{DJM}, \cite{DJKM}); see \cite{FKRW} and \cite{KR} for representation-theoretic considerations.

In the present paper we construct a natural category of topological tensor modules over $\mathfrak{gl}(V)$.
In the recent note \cite{EP}, we have introduced and studied two such anti-equivalent categories $\mathbf{T}_{\mathfrak{gl}(U)}$ and $\widehat{\mathbf{T}}_{\mathfrak{gl}(U)}$ for the discrete vector space $U$. These two categories have universality properties which they inherit from the category $\mathbb{T}_{\mathfrak{gl}}(\infty)$, see \cite{EP}.

One of our objectives is to extend the results of \cite{EP} from the space $U$ to the Tate space $V = U\oplus U^\ast$.
A further objective is to embed the tensor modules of the Lie algebra $\mathfrak{gl}(V)$ into a supersymmetric context and obtain results on categories of tensor supermodules over certain ``semiinfinite'' Lie superalgebras.
Our approach to this problem is based on an idea of V. Serganova who connected the study of tensor modules over the Lie algebras 
$\mathfrak{gl}(\infty)$ and $\mathfrak{o}(\infty)$ or 
$\mathfrak{sp}(\infty)$ with a study of tensor supermodules over the respective Lie superalgebras $\mathfrak{gl}(\infty|\infty)$ and 
$\mathfrak{osp}(\infty|\infty)$ \cite{S}.

Concretely, we construct anti-equivalent monoidal categories 
$\mathbf{T}_{\mathfrak{gl}(V)}$ and $\widehat{\mathbf{T}}_{\mathfrak{gl}(V)}$ of topological tensor 
$\mathfrak{gl}(V)$-modules such that $V, V^\ast$, and $\mathfrak{gl}(V)$ are objects of $\widehat{\mathbf{T}}_{\mathfrak{gl}(V)}$.
The categories $\mathbf{T}_{\mathfrak{gl}(V)}$ and $\widehat{\mathbf{T}}_{\mathfrak{gl}(V)}$ turn out to be equivalent respectively to the categories $\mathbf{T}_{\mathfrak{gl}(U)}$ and $\widehat{\mathbf{T}}_{\mathfrak{gl}(U)}$. This is a consequence of the equivalence of monoidal categories
\[
\begin{tikzcd}
    \mathbf{T}_{\mathfrak{gl}(V)} \ar[r, "\simeq"]& \mathbb{T}_{\mathfrak{gl}(\infty)}
\end{tikzcd}
\]
 which we establish in Section 3. It is an interesting feature that, while the $\mathfrak{gl}(U)$-module $U$ corresponds to the $\mathfrak{gl}(V)$-module $V$ under the equivalence of $\widehat{\mathbf{T}}_{\mathfrak{gl}(U)}$ and $\widehat{\mathbf{T}}_{\mathfrak{gl}(V)}$, the topological vector spaces $U$ and $U^\ast$ are non-isomorphic and the topological vector spaces $V$ and $V^\ast$ are isomorphic.
 We then introduce semiinfinite orthogonal and symplectic Lie algebras $\mathfrak{o}(V)$ and $\mathfrak{sp}(V)$ and respective monoidal categories of topological modules 
 $\mathbf{T}_{\mathfrak{o}(V)}$, $\mathbf{T}_{\mathfrak{sp}(V)}$, $\widehat{\mathbf{T}}_{\mathfrak{o}(V)}$, $\widehat{\mathbf{T}}_{\mathfrak{sp}(V)}$.

 The supersymmetric setting is presented in Section 4. We consider the Tate superspace 
 $V\oplus \Pi V$ and the Lie superalgebras
$\mathfrak{gl}(V\oplus \Pi V)$ and $\mathfrak{osp}(V\oplus \Pi V)$.
We show that the categories of supermodules
$\mathbf{T}_{\mathfrak{gl}(V\oplus\Pi V)}$ and $\widehat{\mathbf{T}}_{\mathfrak{gl}(V\oplus\Pi V)}$ are equivalent respectively to 
$\mathbf{T}_{\mathfrak{gl}(V)}$ and $\widehat{\mathbf{T}}_{\mathfrak{gl}(V)}$.
The monoidal category $\mathbf{T}_{\mathfrak{osp}(V\oplus\Pi V)}$ is equivalent to both monoidal categories 
$\mathbf{T}_{\mathfrak{o}(V)}$ and $\mathbf{T}_{\mathfrak{sp}(V)}$,
while the category $\widehat{\mathbf{T}}_{\mathfrak{osp}(V\oplus\Pi V)}$ is equivalent to both monoidal categories 
$\widehat{\mathbf{T}}_{\mathfrak{o}(V)}$
and $\widehat{\mathbf{T}}_{\mathfrak{sp}(V)}$.
Consequently, the categories $\mathbf{T}_{\mathfrak{o}(V)}$ and $\mathbf{T}_{\mathfrak{sp}(V)}$ are equivalent monoidal categories, and the same holds for $\widehat{\mathbf{T}}_{\mathfrak{o}(V)}$
and $\widehat{\mathbf{T}}_{\mathfrak{sp}(V)}$.
We conclude by a discussion of the universality properties of all categories introduced in the paper.

	\textbf{Acknowledgements.}
	F.E. is a member of the INDAM group GNSAGA; his work has been supported in part by the project of the University of Padova BIRD203834/20.
	The work of I.P. has been  supported in part by DFG grant PE 980/8-1.
	
	\section{Preliminaries on topological vector spaces}
	
	The ground field is $\mathbb{C}$ endowed with the discrete topology. All vector spaces we consider are endowed with a linear topology, and, if not stated explicitly, homomorphisms between topological vector spaces are assumed to be continuous. 
	Vector spaces of at most countable dimension are considered as discrete topological vector spaces; vice versa, all discrete vector spaces considered have at most countable dimension. 
	The dual of a discrete vector space is a linearly compact vector space, i.e. a (countable) projective limit of finite-dimensional vector spaces, topologized with the projective limit topology. The abelian semisimple categories of discrete vector spaces and of linearly compact vector spaces are mutually dual via the functor of taking continuous dual. The intersection of these two categories is the self-dual category of finite-dimensional vector spaces.
	
	A natural next step is to study inductive limits and projective limits of topological vector spaces. The inductive systems we consider are countable and the morphisms involved are continuous with closed image. Dually, the projective systems we consider are countable and the morphisms involved are continuous and surjective. In this manner we regard ind-linearly compact vector spaces and pro-discrete vector spaces as generalizations of linearly compact and discrete vector spaces. 
 
    In \cite{Bei} A. Beilinson introduces tensor products $\otimes^\ast $ and $\otimes^! $ of topological vector spaces.
    In agreement with our prior work \cite{EP}, we denote these tensor products respectively by $\widehat{\otimes}^\ast $ and $\widehat{\otimes}^!$.
    We also set $\otimes := \otimes_{\mathbb{C}}$.

    \begin{defn}
        Let $\mathcal{I}$ be the category whose objects are ind-linearly compact vector spaces and whose morphisms are continuous linear maps.
    \end{defn}

    \begin{defn}
        Let $\mathcal{P}$ be the category whose objects are pro-discrete vector spaces and whose morphisms are continuous linear maps.
    \end{defn}

    \begin{prop}\label{prop:IP}
        \begin{enumerate}[(i)]
            \item The category $\mathcal{I}$ is stable under Beilinson's tensor product $\widehat{\otimes}^\ast$. 
            \item Equipped with $\widehat{\otimes}^\ast$, the category $\mathcal{I}$ is a 
            $\mathbb{C}$-linear symmetric quasi-abelian semisimple monoidal category.
            \item The category $\mathcal{P}$ is stable under Beilinson's tensor product $\widehat{\otimes}^!$. 
            \item Equipped with $\widehat{\otimes}^!$, the category $\mathcal{P}$ is a 
            $\mathbb{C}$-linear symmetric quasi-abelian semisimple monoidal category.
            \item Taking continuous duals yields an antiequivalence of the symmetric monoidal quasi-abelian categories $\mathcal{I}$ and $\mathcal{P}$.
        \end{enumerate}
    \end{prop}
    \begin{proof}
    $(i)$ It is observed in \cite[Remark~12.1]{Po} that for any topological vector spaces $U$ and $W$, the tensor product $U\widehat{\otimes}^\ast W$ is the completion of the space $U\otimes W$ with respect to the $\ast$-topology (see \cite[1.1(a)]{Bei} or \cite[Section 12]{Po} for the definition), and satisfies a universality property regarding continuous multilinear maps. Thus, if $U = \bigcup_i U_i$ and $W = \bigcup_j W_j$ are objects of $\mathcal{I}$, one has $U\otimes W = \bigcup_{i,j}U_i \otimes W_j$ and the $\ast$-topology coincides with the inductive limit topology. Furthermore, the inductive limit of complete topological vector spaces is complete. Hence, one concludes
    $$
    U \widehat{\otimes}^\ast W = \bigcup_{i,j} \ U_i  \widehat{\otimes}^\ast  W_j \ ,
    $$
    and thus $\mathcal{I}$ is closed under $\widehat{\otimes}^\ast$.

    $(ii)$ The fact that $\mathcal{I}$ is quasi-abelian semisimple is proved in \cite[Proposition~2.9]{EP} . Therefore the functor $\widehat{\otimes}^\ast$ is exact and associative by \cite[Proposition~13.4(a)]{Po} , and $\mathcal{I}$ is a symmetric quasi-abelian monoidal semisimple category.

    $(iii)$ It follows from \cite[1.1(b)]{Bei} that for any two pro-discrete topological vector spaces $U = \varprojlim_i U_i$ and $W = \varprojlim_j W_j$, one has 
    $$
    U \widehat{\otimes}^! W = \varprojlim_{i,j} \ U_i \otimes W_j \ .
    $$
     Thus $\mathcal{P}$ is closed under the tensor product operation $\widehat{\otimes}^!$.   

     $(iv)$ The fact that $\mathcal{P}$ is quasi-abelian semisimple is proved in \cite[Proposition~2.9]{EP}. Hence $\widehat{\otimes}^\ast$ is exact, and associative by \cite[Proposition~13.4(c)]{Po}. Consequently $\mathcal{P}$ is a symmetric monoidal quasi-abelian semisimple category.

     $(v)$ The fact that taking continuous duals produces an anti-equivalence between $\mathcal{I}$ and $\mathcal{P}$ as quasi-abelian categories is stated in \cite[Proposition~2.6]{EP} . Moreover, taking continuous dual is clearly a monoidal functor between the category of discrete vector spaces and the category of linearly compact vector spaces. Therefore, for objects $U = \bigcup_i U_i$ and $W = \bigcup_j W_j$ of $\mathcal{I}$,
     one gets
     $$
     (U\widehat{\otimes}^\ast W)^\ast = (\bigcup_i U_i \ \widehat{\otimes}^\ast \ \bigcup_j W_j)^\ast =
     (\bigcup_{i,j} \ U_i \widehat{\otimes}^\ast W_j)^\ast = 
     $$
     $$
     (\bigcup_{i,j} \ U_i \widehat{\otimes} W_j)^\ast =
     \varprojlim_{i,j} \ U_i^\ast \otimes W_j^\ast = \varprojlim_{i,j} \ U_i^\ast \widehat{\otimes}^! W_j^\ast =
     U^\ast \widehat{\otimes}^! W^\ast  \ .
     $$
     One has thus an anti-equivalence of the symmetric monoidal quasi-abelian categories $\mathcal{I}$ and $\mathcal{P}$.
    \end{proof}

    \begin{defn}
        Let $\mathcal{T}$ be the intersection of $\mathcal{I}$ and $\mathcal{P}$ as full subcategories of the category $Top$ of all linearly topologized topological vector spaces. It is the category of locally linearly compact vector spaces, also called Tate vector spaces.
    \end{defn}
 
	Next we present some basic properties of Tate vector spaces.
	\begin{lem}\label{lem:prelim}
	\begin{enumerate}[(i)]
	    \item 
	    Let $W = \bigcup_i W_i$ be an ind-linearly compact vector space defined as the union of an ascending chain of linearly compact vector spaces $\{W_i\}_{i \in \mathbb{Z}}$. Then $W$ is Tate if and only if, for every $i\in \mathbb{Z}$, the space $W_i$ is of finite codimension in $W_{i+1}$.
	    
	    \item 
	    Let $Z = \varprojlim_j Z_j$ be a pro-discrete vector space defined as the projective limit of discrete vector spaces $\{Z_j\}_{j \in \mathbb{Z}}$. Then $Z$ is Tate if and only if, for every $j\in \mathbb{Z}$, the kernel of the surjective map $Z_j \rightarrow Z_{j-1}$ is finite dimensional.
	    
	    \item 
	    The category $\mathcal{T}$ is quasi-abelian semisimple, and self-dual under the functor of taking continuous dual.
	    
	    \item \label{self-dual}
	    Let $V$ be an infinite-dimensional Tate vector space. Then $V$ is self-dual if and only if $V = U\oplus U^\ast$ for an infinite-dimensional discrete vector space $U$.
	\end{enumerate}
	\end{lem}
	
	\begin{proof}
    $(i)$ Let $W = \bigcup_i W_i$ be an ind-linearly compact vector space which is isomorphic to a pro-discrete vector space $Z = \varprojlim_j Z_j$ as a topological vector space, i.e. let $W$ be a Tate vector space. The composition 
    $$
    W_i \rightarrow W \cong Z \rightarrow Z_j
    $$
    is a continuous map from a linearly compact vector space to a discrete vector space; hence its image $Z_{i,j}$ is finite dimensional . One may then express the spaces $W_i$ and $Z_j$ in terms of the $Z_{i,j}$:
    $$
    W_i = \varprojlim_j Z_{i,j} \ \ \mathrm{and} \ \ Z_j = \bigcup_i Z_{i,j} \ .
    $$
    This yields two different expressions of $W$ as double limits:
    $$
    W = \bigcup_i \varprojlim_j Z_{i,j} = \varprojlim_j \bigcup_i Z_{i,j} \ .
    $$
    To understand these double limits, for every $i,j \geq 0$ denote by $K_{i,j}$ the kernel of the surjection $Z_{i,j}\rightarrow Z_{i, j-1}$, and choose complements $Z'_{i,j}$ such that
    $$
    K_{i,j} = Z'_{i,j} \oplus (Z_{i-1, j}\cap K_{i,j}) \ .
    $$
    Then $W_{i}/ W_{i-1} = \prod_j Z'_{i,j}$, and moreover one has the equality
    
    \begin{equation}\label{eq:W}
        W = \bigoplus_{i\geq 0} \prod_{j\geq 0} Z'_{i,j} = \prod_{j\geq 0} \bigoplus_{i\geq 0} Z'_{i,j} \ .
    \end{equation}
    
    Now we note that equality (\ref{eq:W}) holds if and only if, for all but finitely many $i$, one has $Z'_{i,j} = 0$ for all but finitely many $j$. Hence there is no loss in generality in assuming that, for all $i > 0$ only finitely many $Z'_{i,j}$ are nonzero. This means that $W_{i-1}$ is of finite codimension in $W_i$.

    Vice versa, if $W = \bigcup_i W_i$ is such that the linearly compact space $W_{i-1}$ is of finite-codimension in $W_i$ for all $i$, then the subspaces $W_i$ are open in $W$ and form a basis of open neighborhoods of $W$. Hence $W \cong \varprojlim_i W/W_i$ is pro-discrete, and consequently $W$ is a Tate space.

    This proves $(i)$. Statement $(ii)$ is proved along the same lines. Statement $(iii)$ follows from Proposition~\ref{prop:IP}$(v)$.

    Let us prove $(iv)$. Let $U$ be an infinite-dimensional discrete vector space of countable dimension. Then, clearly $U\oplus U^\ast$ is self-dual and Tate. Vice versa, let $V$ be infinite-dimensional self-dual and Tate. Then there is a filtration
    $$
    V = \bigcup_{i\in\mathbb{O}} V_{\geq i}
    $$
    where $V_{\geq i}$ is linearly compact and of codimension one in $V_{\geq i-1}$ and $\mathbb{O}$ is a totally ordered set isomorphic to $\mathbb{Z}$. To make some frmulas more symmetric we fix $\mathbb{O}$ to be $\mathbb{Z}\setminus \{0\}$. We also fix a topological basis $\{v_i\}_{i\in\mathbb{Z}\setminus \{0\}}$ of $V$ for which the linearly compact subspace $V_{\geq i}$ is the closure of $<v_j | j \geq i>$. Then the subspace $V_{<i}$, spanned by the $v_j$ for $j<i$, is discrete and 
    $$
    V = V_{<i} \oplus V_{\geq i} \ .
    $$
    Moreover, one easily sees that there is a topological isomorphism $(V_{\geq i})^\ast \cong V_{<i}$. We conclude that $V = U\oplus U^\ast$, with $U = \bigcup_{i>0}V_{<i}$ discrete of countable dimension.
	
	\end{proof}

	\section{Categories of mixed tensors}
	
	\subsection{The space $V$}
	
	We fix here the notation for the rest of the paper. Let $V$ be a fixed infinite-dimensional self-dual Tate vector space. We denote by $\mathfrak{gl}(V)$ the Lie algebra of continuous endomorphisms of $V$. This is a Lie algebra in the category $\mathcal{P}$.

	By Lemma~\ref{lem:prelim}(\ref{self-dual}), we may fix a decomposition $V = U \oplus U^\ast$, where $U$ is infinite-dimensional discrete. 
	Associated with such decomposition is a nondegenerate symmetric bilinear form
	$$
	A: V\times V \rightarrow \mathbb{C} \ , \ \ 
        A((u, \varphi), (u' , \varphi')) = \varphi(u') + \varphi'(u) \ ,
	$$
	and a nondegenerate antisymmetric bilinear form
	$$
	B: V\times V \rightarrow \mathbb{C} \ , \ \ 
        B((u, \varphi), (u' , \varphi')) = \varphi(u') - \varphi'(u) \ .
	$$
	We denote by $\mathfrak{o}(V) := \mathfrak{o}(V, A)$ the Lie subalgebra of $\mathfrak{gl}(V)$ of those continuous endomorphisms $f: V \rightarrow V$ such that, for all $v,w\in V$
	$$
	A(f(v), w) + A(v, f(w)) = 0 \ .
	$$
	Accordingly, $\mathfrak{sp}(V) := \mathfrak{sp}(V, B)$ is the Lie subalgebra of $\mathfrak{gl}(V)$ of those continuous endomorphisms $f: V \rightarrow V$ such that, for all $v,w\in V$
	$$
	B(f(v), w) + B(v, f(w)) = 0 \ .
	$$
	
	Furthermore, one may fix a basis $\mathcal{U}^- = \{v_{-1}, v_{-2}, \ldots\}$ of the discrete vector space $U$. This yields a topological basis $\mathcal{U}^+ = \{v_1 , v_2 , \ldots \}$ of $U^\ast$ dual to the basis $\mathcal{U}^-$. The union $\mathcal{V} = \{\ldots, v_{-2}, v_{-1}, v_1, v_2, \ldots\}$ is a topological basis of $V$.
	
	By $\mathfrak{h}$ we denote the abelian Lie subalgebra of $\mathfrak{gl}(V)$ consisting of the endomorphisms of $V$ with diagonal matrix with respect to the topological basis $\mathcal{V}$.

    \begin{defn}
        Let $\mathfrak{h}_1$ be the abelian Lie algebra intersection of $\mathfrak{h}$ with $\mathfrak{o}(V)$, or equivalently with $\mathfrak{sp}(V)$; it consists of the diagonal matrices $D = (a_{i,i})_{i\in \mathbb{Z}\setminus\{0\}}$ with respect to the basis $\mathcal{V}$, such that $a_{i,i} = -a_{-i, -i}$.
    \end{defn}

    \begin{defn}
        Let $\mathfrak{g}(V)$ be any one of the topological Lie algebras $\mathfrak{gl}(V), \mathfrak{o}(V), \mathfrak{sp}(V)$, and let $\mathfrak{h}_{\mathfrak{g}(V)} := \mathfrak{h}$ in the case $\mathfrak{g}(V) = \mathfrak{gl}(V)$, and  $\mathfrak{h}_{\mathfrak{g}(V)} := \mathfrak{h}_1$ in the cases $\mathfrak{g}(V) = \mathfrak{o}(V)$ or $\mathfrak{g}(V) = \mathfrak{sp}(V)$.
    \end{defn}

	\subsection{Topological vector spaces of mixed tensors}

	\begin{defn}
	The topological spaces of mixed tensors of $V$ are the spaces
	$$
	\mathbf{V}^{p,q} \ := \ V^{\widehat{\otimes}^\ast \! p}\  \widehat{\otimes}^\ast (V^\ast)^{\widehat{\otimes}^\ast \! q}
	$$
	and
	$$
	\widehat{\mathbf{V}}^{p,q} \ := \ V^{\widehat{\otimes}^! \! p\ } \widehat{\otimes}^! (V^\ast)^{\widehat{\otimes}^! \! q} 
	\ .
	$$
	\end{defn}
	By definition, the space $V$ is a representation of the Lie algebra $\mathfrak{gl}(V)$. Therefore, by functoriality, all topological spaces of  mixed tensors also are representations, or modules, of $\mathfrak{gl}(V)$. Furthermore, by restriction, these are  modules also over $\mathfrak{o}(V)$ and $\mathfrak{sp}(V)$.
	\begin{lem}
	The space $\mathbf{V}^{p,q}$ is ind-linearly compact. The space $\widehat{\mathbf{V}}^{q,p}$ is the continuous dual of $\mathbf{V}^{p,q}$ and is thus pro-discrete.
	\end{lem}
	\begin{proof}
Follows directly from Proposition~\ref{prop:IP} $(i)$, $(v)$.
	\end{proof}

    \begin{lem}\label{lem:vvast}
       $V^\ast$ is isomorphic to $V$ as a topological 
       $\mathfrak{o}(V)$-module, and also as a topological
            $\mathfrak{sp}(V)$-module.     
    \end{lem}
    \begin{proof}
We carry out the argument for the bilinear form $A$, the argument for $B$ being analogous. The map $v\mapsto A(v,-)$ is a linear map from $\Phi_A : V\rightarrow V^\ast$, and one checks that it is bijective and bicontinuous, for example from the decomposition 
$V=U\oplus U^\ast$. Furthermore,
if $\varphi$ is the linear form $A(v,-)$, then $\Phi_A (v) = \varphi$, and if 
$X\in \mathfrak{o}(V)$ one gets 
\[
\Phi_A (Xv) = X\Phi_A(v) \ .
\]
Thus the isomorphism $\Phi_A$ is $\mathfrak{o}(V)$-equivariant.
    \end{proof}
	
	\subsection{$\mathfrak{h}_{\mathfrak{g}(V)}$-module structure and weight part}
	Let $W$ be an $\mathfrak{h}_{\mathfrak{g}(V)}$-module. Recall that an element $\chi\in \mathfrak{h}_{\mathfrak{g}(V)}^\ast$ is a \emph{weight of} $W$ if
	$$
	W^\chi = \left\{ w\in W \ | \ \ tw = \chi(t)w \ ,  \ \ 
 \forall t\in\mathfrak{h}_{\mathfrak{g}(V)}\right\} \neq 0 \ .
	$$
	The space $W^\chi$ is the \emph{$\chi$-weight space of} $W$.
	The sum of all weight spaces of $W$ is the \textit{weight part} $W^{wt}$ \emph{of} $W$; it is the largest semisimple $\mathfrak{h}_{\mathfrak{g}(V)}$-submodule of $W$.

 The weights $\varepsilon_k\in \mathfrak{h}^\ast$ for $k\in \mathbb{Z} \setminus \{0\}$ are by definition the weights of the $\mathfrak{gl}(V)$-module $V$, and $\dim V^{\varepsilon_k} = 1$ for all $k$. In particular, the weights of $V_{\geq i}$ are all $\varepsilon_k$ for $k\geq i$. Furthermore, the vectors $v_k ^\ast$ form a topological basis of $V^\ast$ which is dual to $\mathcal{V}$, and $v_k ^\ast$ has weight $-\varepsilon_k$.  Next, the linear forms $\chi_k := (\varepsilon_k)_{|\mathfrak{h}_{\mathfrak{g}(V)}}$ form a basis of $\mathfrak{h}_{\mathfrak{g}(V)}^\ast$. If $\mathfrak{g}(V) = \mathfrak{o}(V)$ or $\mathfrak{g}(V) = \mathfrak{sp}(V)$, we assume $k>0$ since here $\chi_{-k} = -\chi_k$.

 	\begin{lem} \label{lem: weightpart}
	The following statements hold:
	\begin{enumerate}[(i)]
	    \item \label{product}
	    The subspaces 
     \[
    \mathbf{V}^{p,q}_{i} :=
     V_{\geq i} ^{\widehat{\otimes}^\ast p}\ \widehat{\otimes}^\ast ((V_{\leq -i})^\ast)^{\widehat{\otimes}^\ast q}
     \]
     are linearly compact $\mathfrak{h}_{\mathfrak{g}(V)}$-submodules of the $\mathfrak{g}(V)$-module $\mathbf{V}^{p,q}$. 
     For any weight $\chi\in \mathfrak{h}_{\mathfrak{g}(V)}^\ast$ of $\mathbf{V}^{p,q}_{i}$, the weight space 
     $(\mathbf{V}^{p,q}_{i})^\chi$ is finite dimensional, and 
	$$
	\mathbf{V}^{p,q}_{i} = \prod_{\chi} (\mathbf{V}^{p,q}_{i})^\chi
	$$
	with only a countable number of weights occurring. Furthermore, 
	
	$$
	(\mathbf{V}^{p,q})^{wt} = V^{p,q} := (V^{wt})^{\otimes p} \otimes (V^{\ast wt})^{\otimes q} \ .
	$$
	
	\item \label{bij}
	There is an order-preserving bijection between closed $\mathfrak{h}_{\mathfrak{g}(V)}$-stable subspaces of $\mathbf{V}^{p,q}$ and $\mathfrak{h}_{\mathfrak{g}(V)}$-stable subspaces of ${V}^{p,q}$, given in one direction by taking
	weight part, and in the other direction by taking closure. In particular $(\mathbf{V}^{p,q})^{wt}$ is dense in
	$\mathbf{V}^{p,q}$.
	\end{enumerate}

	\end{lem}
	
	\begin{proof} 
    Observe that the space $\mathbf{V}^{p,q}_{i}$ is linearly compact and may be realized as the following inverse limit
	    $$
	    \mathbf{V}^{p,q}_{i} = 
	    \varprojlim_r 	
     V_{[i, r]} ^{\otimes p}
     \otimes 
     (V_{[-r, -i]} ^{\otimes q})^* \ , \ \ r > |i|  \ .
	    $$
	    The canonical map 
	    $$
	    V_{[i, r]} ^{\otimes p}
     \otimes 
     (V_{[-r, -i]} ^{\otimes q})^*
     \rightarrow
     V_{[i, r-1]} ^{\otimes p}
     \otimes 
     (V_{[-(r-1), -i]} ^{\otimes q})^*
	    $$
	    is a surjective homomorphism of $\mathfrak{h}_{\mathfrak{g}(V)}$-modules. It has a unique right inverse due to the fact that it is an isomorphism when restricted to the direct sum of the weight subspaces of 
	    $V_{[i, r]} ^{\otimes p}
     \otimes 
     (V_{[-r, -i]} ^{\otimes q})^*$ 
     corresponding to weights of 
	    $V_{[i, r-1]} ^{\otimes p}
     \otimes 
     (V_{[-(r-1),-i])} ^{\otimes q})^*$. 
     It follows that all weight spaces relative to $\mathfrak{h}_{\mathfrak{g}(V)}$ of 
	    $\mathbf{V}^{p,q}_{i}$ are finite dimensional. Furthermore, the canonical map
	    $$
	    \prod_\chi (\mathbf{V}^{p,q}_{i})^\chi =
	    \varprojlim_\chi (\mathbf{V}^{p,q}_{i})^\chi
	    \rightarrow 
	    \varprojlim_r 
     V_{[i, r]} ^{\otimes p}
     \otimes 
     (V_{[-r, -i]} ^{\otimes q})^*
	    $$
	    	is an isomorphism of topological vector spaces. The weights of 
	$\mathbf{V}^{p,q}_{i}$ are of the form $\chi = \sum_{k}n_k \chi_k$ (note that $\chi_k = \varepsilon_k$ for $\mathfrak{g}(V) = \mathfrak{gl}(V)$), 
 where the integers $n_k$ satisfy
\[
\sum_{k}|n_k | \leq p+q \ .
\]

	Furthermore, the above implies
	\begin{equation*}
	    \begin{split}
	    (\mathbf{V}^{p,q})^{wt} & = \bigcup_{i,j} (\mathbf{V}^{p,q}_{i})^{wt} =
	\bigcup_{i,r} (V_{[i, r]} ^{\otimes p}
     \otimes 
     (V_{[-r, -i]} ^{\otimes q})^*)^{wt} \\
	& = 
	\bigcup_{i, r} (V_{[i, r]} ^{\otimes p}
     \otimes 
     (V_{[-r, -i]} ^{\otimes q})^*) 
	=  V ^{\otimes p}\otimes V_\ast^{\otimes q} = {V}^{p,q} \ ,
	\end{split}
	\end{equation*}
		where the injection 
	$V_{[i, r-1]} ^{\otimes p}
     \otimes 
     (V_{[-(r-1), -i]} ^{\otimes q})^*
	\hookrightarrow
	V_{[i, r]} ^{\otimes p}
     \otimes 
     (V_{[-r, -i]} ^{\otimes q})^*$ is the unique homomorphism of $\mathfrak{h}_{\mathfrak{g}(V)}$-modules which is right inverse to the canonical surjection 
	$V_{[i, r]} ^{\otimes p}
     \otimes 
     (V_{[-r, -i]} ^{\otimes q})^*
	\rightarrow
	V_{[i, r-1]} ^{\otimes p}
     \otimes 
     (V_{[-(r-1), -i]} ^{\otimes q})^*$.
	This concludes the proof of $(i)$.
	
	Let us now prove $(ii)$. Observe that, for each $i$, there is an order-preserving bijection between the closed $\mathfrak{h}_{\mathfrak{g}(V)}$-stable subspaces of 
    $\prod_\chi (\mathbf{V}^{p,q}_{i})^\chi $ 
    and the $\mathfrak{h}_{\mathfrak{g}(V)}$-stable subspaces of its weight part $\bigoplus_\chi (\mathbf{V}^{p,q}_{i})^\chi$. If $L$ is an $\mathfrak{h}_{\mathfrak{g}(V)}$-stable subspace of ${V}^{p,q}$, then $L = \bigcup_{i} L_{i}$ where $L_{i} = L \cap (\mathbf{V}^{p,q}_{i})^{wt}$. 
	Furthermore, one checks that $\overline{L_{i}} = \overline{L_{i-1}}\cap (\mathbf{V}^{p,q}_{i})$.
	Therefore the closure $\overline{L}$ of $L$ in $\mathbf{V}^{p,q}$ is equal to $\bigcup_{i} \overline{L_{i}}$, where $\overline{L_{i}}$ is the closure of $L_{i}$ in $\mathbf{V}^{p,q}_{i}$. Thus, taking the weight part gives rise to an order-preserving bijection between closed $\mathfrak{h}_{\mathfrak{g}(V)}$-stable subspaces of $\mathbf{V}^{p,q}$ and $\mathfrak{h}_{\mathfrak{g}(V)}$-stable subspaces of $V^{p,q}$. The inverse bijection is given by taking the closure. In particular, one has $\overline{(\mathbf{V}^{p,q})^{wt}}=\mathbf{V}^{p,q}$, so $(\mathbf{V}^{p,q})^{wt}$ is dense in $\mathbf{V}^{p,q}$. This concludes the proof of $(ii)$.
	\end{proof}

	\subsection{The categories $\mathbf{T}_{\mathfrak{g}(V)}$ and $\widehat{\mathbf{T}}_{\mathfrak{g}(V)}$}

        The definition of the categories $\mathbf{T}_{\mathfrak{gl}(V)}$, $\mathbf{T}_{\mathfrak{o}(V)}$, $\mathbf{T}_{\mathfrak{sp}(V)}$ runs parallel, therefore it is convenient to have a single definition of a category $\mathbf{T}_{\mathfrak{g}(V)}$ of topological $\mathfrak{g}$-modules, where $\mathfrak{g}(V)$ is one of the topological Lie algebras $\mathfrak{gl}(V)$, $\mathfrak{o}(V)$ or $\mathfrak{sp}(V)$.
        Analogously, it is possible to give a uniform definition of the categories $\widehat{\mathbf{T}}_{\mathfrak{gl}(V)}$, $\widehat{\mathbf{T}}_{\mathfrak{o}(V)}$, $\widehat{\mathbf{T}}_{\mathfrak{sp}(V)}$. 
	
	\begin{defn}\label{def:tg}
		The objects of the category $\mathbf{T}_{\mathfrak{g}(V)}$ are topological vector spaces of the form $Z/Q$, where $Z$ is a closed $\mathfrak{g}(V)$-stable subspace of a finite direct sum of topological vector spaces of the form $\mathbf{V}^{p,q}$, and $Q$ is a closed $\mathfrak{g}(V)$-stable subspace of $Z$. The morphisms in the category $\mathbf{T}_{\mathfrak{g}(V)}$ are $\mathfrak{g}(V)$-equivariant continuous linear maps.
	\end{defn}

        \begin{rmk}
            Lemma~\ref{lem:vvast} implies that in the case of
            $\mathfrak{o}(V)$ or $\mathfrak{sp}(V)$ in Definition~\ref{def:tg}
            it is sufficient to consider only spaces $\mathbf{V}^{p,q}$ with $q=0$.
        \end{rmk}

        \begin{lem}\label{T-split}
		Let $W$ be an object of $\mathbf{T}_{\mathfrak{g}(V)}$ and let $Z$ be a closed $\mathfrak{h}_{\mathfrak{g}(V)}$-stable subspace of $W$. Then there is a closed $\mathfrak{h}_{\mathfrak{g}(V)}$-stable subspace $Q$ of $W$ such that the canonical map $Z\oplus Q \rightarrow W$ is an $\mathfrak{h}_{\mathfrak{g}(V)}$-equivariant topological isomorphism. 
	\end{lem}
	
	\begin{proof}	    
	    It is clear from Definition~\ref{def:tg} that it suffices to consider the case $W = \mathbf{V}^{p,q}$. Here, the filtration $\{\mathbf{V}^{p,q}_{i}\}$ induces an $\mathfrak{h}_{\mathfrak{g}(V)}$-stable filtration by closed linearly compact subspaces 
	    $W = \bigcup_{i} W_{i}$ which, by Lemma~\ref{lem: weightpart} $(i)$, are isomorphic to direct products of finite-dimensional weight spaces. Analogously, the closed subspace $Z$ has the induced filtration $Z = \bigcup_{i} Z_{i}$.  For every weight $\chi$ of $W$, one may choose compatible supplementary subspaces $(Q_{i})^\chi$ of $(Z_{i})^\chi$ in $(W_{i})^\chi$. Then the subspaces $Q_{i} := \prod_\chi (Q_{i})^\chi$ are closed and $\mathfrak{h}_{\mathfrak{g}(V)}$-stable, and  there are isomorphisms $Z_{i} \oplus Q_{i} \cong W_{i}$ of topological $\mathfrak{h}_{\mathfrak{g}(V)}$-modules. Consequently, $Q:=Q_{i}$  is a closed and $\mathfrak{h}_{\mathfrak{g}(V)}$-stable subspace of $W$, and the canonical morphism $Z\oplus Q \rightarrow W$ is an isomorphism of topological $\mathfrak{h}_{\mathfrak{g}(V)}$-modules. This concludes the proof.
	\end{proof}

        \begin{defn}\label{def:gvw}
            Let $\mathfrak{g}(\infty)$ be the subalgebra of $\mathfrak{g}(V)$ stabilizing $V^{wt}$.
        \end{defn}

        \begin{defn}\label{tgvw}
            The objects of the category $\mathbb{T}_{\mathfrak{g}(\infty)}$ are vector spaces of the form $Z'/Q'$, where $Z'$ is a $\mathfrak{g}(\infty)$-stable subspace of a finite direct sum of vector spaces of the form $(\mathbf{V}^{p,q})^{wt}$, and $Q'$ is a $\mathfrak{g}(\infty)$-stable subspace of $Z'$. 
            The morphisms in the category $\mathbb{T}_{\mathfrak{g}(\infty)}$ are $\mathfrak{g}(\infty)$-equivariant linear maps.
        \end{defn}

        \begin{lem}
	    Taking the weight part is a functor $(\phantom{a})^{wt}: \mathbf{T}_{\mathfrak{g}(V)} \rightarrow \mathbb{T}_{\mathfrak{g}(\infty)}$.
	\end{lem}
        \begin{proof}
         By Lemma~\ref{lem: weightpart}, the weight part of $\mathbf{V}^{p,q}$ is $V^{p,q}$. Note that $V^{p,q}$ is not $\mathfrak{g}(V)$-stable, but is stable by the dense Lie subalgebra $\mathfrak{g}(\infty)$. Let $W\cong Z/Q$ be an object of $\mathbf{T}_{\mathfrak{g}(V)}$ where $Q\subset Z$ are two closed $\mathfrak{g}(V)$-stable subspaces of a finite direct sum  $\bigoplus_i \mathbf{V}^{p_i,q_i}$. Then the weight parts $Q^{wt}$ and $Z^{wt}$ are $\mathfrak{g}(\infty)$-stable subspaces of $\bigoplus_i V^{p_i,q_i}$, and Lemma~\ref{T-split} implies that $W^{wt} \cong Z^{wt}/Q^{wt}$. Thus the weight part $W^{wt}$ is an object of the category $\mathbb{T}_{\mathfrak{g}(\infty)}$. Finally, it is clear that morphisms in $\mathbf{T}_{\mathfrak{g}(V)}$ restrict to $\mathfrak{g}(\infty)$-equivariant morphisms between weight parts. This proves the statement.  
        \end{proof}

        \begin{lem}\label{lem:full}
	The map 
	$$\Phi:\Hom_{\mathbf{T}_{\mathfrak{g}(V)}}(\mathbf{V}^{p,q}, \mathbf{V}^{p', q'})
	\rightarrow
	\Hom_{\mathbb{T}_{\mathfrak{g}(\infty)}}((\mathbf{V}^{p,q})^{wt}, (\mathbf{V}^{p', q'})^{wt})
	$$
	induced by the functor $(\phantom{a})^{wt}$, is an isomorphism of vector spaces.
	\end{lem}
 \begin{proof}
     Since  by Lemma~\ref{lem: weightpart} $(ii)$ the subspace $(\mathbf{V}^{p,q})^{wt}$ is dense in $\mathbf{V}^{p,q}$, the map $\Phi$ is injective. Furthermore, by \cite[Section 6]{DPS}, the vector space $\Hom_{\mathbb{T}_{\mathfrak{gl}(\infty)}}((\mathbf{V}^{p,q})^{wt}, (\mathbf{V}^{p', q'})^{wt})$ is generated by compositions of contractions and permutations. These morphisms clearly extend  to $\mathbf{V}^{p,q}$ by continuity, and hence are in the image of the map $\Phi$. Thus $\Phi$ is also surjective.
 \end{proof}

        Recall that the definition of a strict morphism in a quasi-abelian category is given in \cite[1.1.7]{Sch},  or see \cite[Section~2]{EP}.
        
        \begin{prop}\label{prop:strict}
	    Let $$W = \bigoplus_{k=1}^r \mathbf{V}^{p_k , q_k} \ \ \mathrm{and} 
	    \ \ W' = \bigoplus_{k'=1}^{r'} \mathbf{V}^{p'_{k'} , q'_{k'}}. $$
	    Then every $f\in\Hom_{\mathbf{T}_{\mathfrak{g}(V)}}(W, W')$ is strict.
	    Moreover, if $Q$ is a closed $\mathfrak{g}(V)$-stable subspace of $W$, every restriction $f_{|Q}$ is also strict.
	\end{prop}
        \begin{proof}
        Let $f\in\Hom_{\mathbf{T}_{\mathfrak{g}(V)}}(W, W')$ and let $Z$ denote the image of $f$. 
        By \cite[Lemma~2.10]{EP}, to show that $f$ is strict it suffices to prove that $Z$ is a closed subspace of $W'$. 
	
	Consider 
 $W_{i} := \bigoplus_{k=1}^r \mathbf{V}_{i} ^{p_k, q_k}$
  and 
  $W'_{i'} := \bigoplus_{k'=1}^{r'} \mathbf{V}_{i'} ^{p'_k, q'_k}$. 
 The topologies on $W$ and $W'$ are the inductive limit topologies on $\bigcup_{i} W_{i}$ and $\bigcup_{i'} W'_{i'}$ respectively. Thus, if $Z_{i'} := Z \cap W'_{i'}$ one has $Z = \bigcup_{i'} Z_{i'}$; moreover $Z$ is closed in $W'$ if and only if, for every $i'$, the subspace $Z_{i'}$ is closed in 
 $W'_{i'}$.
	
	Let us fix $i'$ and prove that $Z_{i'}$ is closed in $W'_{i'}$. Set $Z_{i', i}:=f(W_{i})\cap W'_{i'}$. Then $Z_{i'} = \bigcup_{i} Z_{i', i}$, and, by \cite[Lemma~2.3 $(iii)$]{EP}, the subspace $Z_{i'}$ is closed in $W'_{i'}$ if and only if there exists $i$ such that $Z_{i', i} = Z_{i'}$.
	
	We now show that indeed there is an $i$ such that $Z_{i', i} = Z_{i'}$, thus proving the proposition. By Lemma~\ref{lem: weightpart} $(i)$, the the $\chi$-weight spaces of $W_{i}$ and $W'_{i'}$ are finite dimensional for any $\chi\in\mathfrak{h}_{\mathfrak{g}(V)}^\ast$, and
	$$
	\overline{Z_{i'}}= \prod_\chi (Z_{i'})^\chi  \ \ \mathrm{and} \ \ Z_{i', i} =\prod_\chi (Z_{i', i})^\chi.
	$$

 Let us treat the case $\mathfrak{g}(V) = \mathfrak{gl}(V)$.
	By $S$ we denote the finitary symmetric group on countably many letters, i.e., $S=\bigcup_n S_n$.
	Then $S$ acts on $V$ by permuting the elements of the basis $\mathcal{V}$, and $S$ acts contragrediently on $V^\ast$. Therefore $S$ acts also on $W$ and $W'$. Lemma~\ref{lem:full} and \cite[Lemma 6.1]{DPS} imply that the map $f$ is $S$-equivariant. Observe that $W'_{i'}$ is stable under the action of the subgroup $S'$ of $S$ which  fixes pointwise the set 
$\{v_i \}$ for $|i| \leq N := |i'| $.
Furthermore, since $S$ normalizes $\mathfrak{h}_{\mathfrak{g}(V)} = \mathfrak{h}$, it follows that the action of $S'$ permutes the weights and, accordingly, the respective weight spaces. Denote $r = N + p + q$, 
and observe that the weights $\chi$ for which $(W'_{i'})^{\chi}\neq 0$ form finitely many $S'$-orbits and each orbit contains a weight of 
 $V_{[i', r]} ^{\otimes p}
     \otimes 
     (V_{[-\eta, -i']} ^{\otimes q})^*$. 
     Let $i$ be such that $(Z_{i'})^\chi = (Z_{i', i})^\chi$ for every weight $\chi$ of 
	$V_{[i', r]} ^{\otimes p}
     \otimes 
     (V_{[-r, i']} ^{\otimes q})^*$.
	
	Let us show that this implies $(Z_{i'})^\psi = (Z_{i', i})^\psi$ for any weight $\psi$, and thus ultimately $Z_{i'} = Z_{i', i}$. Indeed, for $k<l$, let $E_{k,l}\in\mathfrak{gl}(V)$ be the endomorphism sending $v_k$ to $v_l$, and sending all other basis vectors to $0$. Observe that $E_{k,l}(W'_{i'})\subset W'_{i'}$, and that $E_{k,l}(W_{j})\subset W_{j}$ for every $j$; since $f$ is $\mathfrak{gl}(V)$-equivariant, one has also $E_{k,l}(Z_{i', i})\subset Z_{i', i}$.
	
	Let us proceed by contradiction.

 Suppose $\psi = \sum a_l \varepsilon_l$ 
 is a weight for which 
	$(Z_{i'})^\psi \neq (Z_{i', i})^\psi$,
 and moreover that the number of nonzero coefficients $a_l$ with $|l| > r$ is minimal over all such weights.
  By the hypothesis on $i$, there is at least one 
    $l$ such that  $|l|> r$ and $|a_l| > 0$, and there is at least one $k$ such that $|i'| < k \leq r$ and $a_k = 0$. 
	
	Let $\sigma\in S'$ be the transposition exchanging $k$ and $l$. Let $\theta = \ \sigma ( \psi )$. By $S$-equivariance, one has $\dim(Z_{i'}^{\theta}) = \dim(Z_{i'}^{\psi})$. Furthermore, one checks that the operator $E_{k,l}^{|a_k|}$ sends  $(W'_{i'})^{\theta}$ isomorphically to 
 $(W'_{i'})^{\psi}$. Then, by the $\mathfrak{gl}(V)$-equivariance of $f$, the same operator sends isomorphically $(Z_{i', i})^{\theta}$ to $(Z_{i', i})^{\psi}$ and $(Z_{i', i})^{\theta}$ to $(Z_{i', i})^{\psi}$. Since in the decomposition of $\theta$ there appears one less $\varepsilon_l$ with $|l|>r$ we have a contradiction, and it implies $(Z_{i'})^{\theta}=(Z_{i', i})^{\theta}$ and thus also $(Z_{i'})^{\psi}=(Z_{i', i})^{\psi}$. 
 This proves that $Z_{i'} = Z_{i', i}$, and hence $f$ is strict in the case $\mathfrak{g}(V) = \mathfrak{gl}(V)$.

Let now $Q$ be a closed $\mathfrak{gl}(V)$-stable subspace of $W$. Since by Lemma~\ref{lem: weightpart} one has $Q = \overline{Q^{wt}}$, it follows that $Q$ is $S$-stable. If $Q_i := Q\cap W_i$, then $Q=\bigcup_i Q_i$ is a presentation of $Q$ as inductive limit of linearly compact closed $\mathfrak{h}_{\mathfrak{g}(V)}$-submodules. The above arguments can be repeated with $Z=f(Q)$, $Z_j = Z\cap W'_j$, $Z_{j,i}=f(Q_i)\cap Z_j$, to prove that $f_{|Q}$ has closed image. Hence $f_{|Q}$ is strict.

 The case $\mathfrak{g}(V) = \mathfrak{o}(V)$ or $\mathfrak{g}(V) = \mathfrak{sp}(V)$ is dealt with analogously. Indeed in this case, the group $S$ acts on $U$ by permuting the basis elements $\ldots, u_{-2}, u_{-1}$, and contragrediently on $U^\ast$ permuting accordingly the topological basis $u_1, u_2 , \ldots$. Hence $S$ acts on $V = U\oplus U^\ast$ in a $\mathfrak{g}(V)$-equivariant way, and ultimately on $W$ and $W'$. Similarly to the case $\mathfrak{g}(V) = \mathfrak{gl}(V)$, results in \cite[Section 6]{DPS} show that $f$ is $S$-equivariant. Let $S'$ be the subgroup fixing the vectors $v_{-N}, \ldots, v_N$ for $N = |i'|$. One uses again the $S'$-action on the weights of $\mathfrak{h}(V)$ and on the various weight spaces to show that $Z_{i'} = Z_{i', i}$ for an appropriate $i$. One must only  substitute $E_{k,l}$ with $E_{k,l}-E_{-k, -l}$, $0<k <l$, for $\mathfrak{g}(V) = \mathfrak{o}(V)$, and with $E_{k,l}+E_{-k, -l}$, $0<k <l$, for $\mathfrak{g}(V) = \mathfrak{sp}(V)$.       
        \end{proof}

        \begin{prop}\label{prop:sub}
	Every object of $\mathbf{T}_{\mathfrak{g}(V)}$ is isomorphic to a subobject of a finite direct sum of $\mathbf{V}^{p,q}$-s.
	\end{prop}

        \begin{proof}
            Let $W = Z/Q$ where $Q\subset Z$ are closed $\mathfrak{g}(V)$-equivariant subspaces of a finite direct sum $M = \bigoplus_j \mathbf{V}^{p_i,q_i} $. It follows from Lemma~\ref{T-split} that $W^{wt}\cong Z^{wt}/Q^{wt}$. Furthermore, by \cite[Proposition 4.5]{DPS}, any object of $\mathbb{T}_{\mathfrak{g}(\infty)}$ is isomorphic to a subobject of an injective object of the form $\bigoplus_j V^{r_j,s_j} $. Let $\varphi: W^{wt}\rightarrow A$ be a $\mathfrak{g}(V)$-equivariant isomorphism, where $A$ is a $\mathfrak{g}(V)$-stable subspace of 
	$( M' ) ^{wt}= \bigoplus_j V^{r_j,s_j} $  for $M' = \bigoplus_j \mathbf{V}^{r_j,s_j} $. Let $Z'$ be the closure of $A$ in $M'$. Moreover, by the injectivity of $( M' ) ^{wt} $ and Lemma~\ref{lem:full}, it follows that $\varphi$ is the restriction of a $\mathfrak{g}(V)$-equivariant map $\Tilde{\varphi}: M^{wt} \rightarrow (M')^{wt}$ which extends to a 
	$\mathfrak{g}(V)$-equivariant map $f: M \rightarrow M'$. By Proposition~\ref{prop:strict}, the subspace $f(Z)$ is closed and is thus equal to $Z'$. Now $f$ induces a bijective continuous linear map between $W$ and $Z'$, which is a topological isomorphism by \cite[Lemma~2.7(a)]{EP}. This proves the proposition.
        \end{proof}

        \begin{defn}
		The objects of the category $\widehat{{\mathbf T}}_{\mathfrak{g}(V)}$ are topological vector spaces of the form $Z/Q$, where $Z$ is a closed $\mathfrak{g}(V)$-stable subspace of a finite direct sum of topological vector spaces of the form $\widehat{\mathbf{V}}^{p,q}$, and $Q$ is a closed $\mathfrak{g}(V)$-stable subspace of $Z$. The morphisms in the category $\widehat{\mathbf{T}}_{\mathfrak{g}(V)}$ are $\mathfrak{g}(V)$-equivariant continuous linear maps.
	\end{defn}

        \begin{thm}\label{thm:main}
	The following statements hold:
	\begin{enumerate}[(i)]
	    \item \label{uno}
	    The categories $\mathbf{T}_{\mathfrak{g}(V)}$ and $\widehat{{\mathbf T}}_{\mathfrak{g}(V)}$ are symmetric monoidal abelian subcategories of $\mathcal{I}$ and $\mathcal{P}$ respectively.
	    \item \label{due}
	    The duality between $\mathcal{I}$ and $\mathcal{P}$ restricts to a duality between $\mathbf{T}_{\mathfrak{g}(V)}$ and $\widehat{{\mathbf T}}_{\mathfrak{g}(V)}$.
	    \item \label{tre}
	    The weight part functor
	    $$
	    (\phantom{a})^{wt}: \mathbf{T}_{\mathfrak{g}(V)} \rightarrow \mathbb{T}_{\mathfrak{g}(\infty)}
	    $$
	    is an equivalence of symmetric monoidal abelian categories.
	\end{enumerate}
	\end{thm}
        \begin{proof}
        The fact that the categories $\mathbf{T}_{\mathfrak{g}(V)}$ and $\widehat{{\mathbf T}}_{\mathfrak{g}(V)}$ are symmetric monoidal categories is a consequence of the isomorphisms 
        \[
        \mathbf{V}^{p,q}\widehat{\otimes}^\ast \mathbf{V}^{p',q'} \cong \mathbf{V}^{p+p',q+q'} \ \mathrm{and}\ \ \widehat{\mathbf{V}}^{p,q}\widehat{\otimes}^! \widehat{\mathbf{V}}^{p',q'} \cong \widehat{\mathbf{V}}^{p+p',q+q'} \ .
        \]
            From the description in \cite[Prop.~2.9]{EP} of kernels and cokernels in $\mathcal{I}$ and $\mathcal{P}$, it follows that the categories $\mathbf{T}_{\mathfrak{g}(V)}$ and $\widehat{{\mathbf T}}_{\mathfrak{g}(V)}$ are closed under taking kernels and cokernels, respectively in $\mathcal{I}$ and $\mathcal{P}$. Thus $\mathbf{T}_{\mathfrak{g}(V)}$ and $\widehat{{\mathbf T}}_{\mathfrak{g}(V)}$ inherit the quasi-abelian structure.
	Furthermore, a linear map $f$ is $\mathfrak{g}(V)$-equivariant if and only if its dual $f^\ast$ is $\mathfrak{g}(V)$-equivariant. This implies that the duality stated in Proposition~\ref{prop:IP} restricts to a duality between the quasi-abelian categories $\mathbf{T}_{\mathfrak{g}(V)}$ and $\widehat{{\mathbf T}}_{\mathfrak{g}(V)}$, and hence proves 
	($\ref{due}$).
	
	To prove ($\ref{uno}$), observe that a quasi-abelian category is abelian if and only if every morphism is strict (see \cite[Remark~4.7 and Remark~4.9]{Buh}). By duality, it suffices to prove that every morphism in $\mathbf{T}_{\mathfrak{g}(V)}$ is strict. 
	
	Let $f:W\rightarrow W'$ be a morphism in $\mathbf{T}_{\mathfrak{g}(V)}$. By Proposition~\ref{prop:sub}, one may suppose that $W$ and $W'$ are submodules of respective modules $M$ and $M'$ which are finite direct sums of $\mathbf{V}^{p,q}$-s. By \cite[Proposition~4.5]{DPS} and Lemma~\ref{lem:full}, one gets that the $\mathfrak{g}(V)$-equivariant continuous linear maps from $W$ to $W'$ are restrictions of $\mathfrak{g}(V)$-equivariant continuous linear maps from $M$ to $M'$, and such restrictions are strict by Proposition~\ref{prop:strict}. Thus ($\ref{uno}$) is proved.
	
	Let us prove ($\ref{tre}$). We keep the notations $W, W', M, M'$ from $(i)$. Then one has a commutative diagram of linear operators
	\[
 \begin{CD}
	\Hom_{\mathbf{T}_{\mathfrak{g}(V)}}(M, M')@>{a}>> \Hom_{\mathbf{T}_{\mathfrak{g}(V)}}(W, M') \\
	@V{c}VV @V{d}VV \\
	\Hom_{\mathbb{T}_{\mathfrak{g}(\infty)}}(M^{wt}, (M')^{wt})@>{b}>> \Hom_{\mathbb{T}_{\mathfrak{g}(\infty)}}(W^{wt}, (M')^{wt}) \ .
	\end{CD} 
    \]

	The map $c$ is bijective by Lemma~\ref{lem:full}; the map $b$ is surjective by injectivity of $(M')^{wt}$ (\cite[Proposition~4.5]{DPS}). It follows that $d$ is surjective.
	Since $W^{wt}$ is dense in $W$, the map $d$ is also injective, and hence $d$ is an isomorphism.
	
	Assume $W'$  is the kernel of a homomorphism
	$g\in \Hom_{\mathbf{T}_{\mathfrak{g}(V)}}(M', N')$, where $N'$ is a finite direct sum of $\mathbf{V}^{p,q}$-s. This leads to the commutative diagram
	\[
        \begin{CD}
	0 @>>>\Hom_{\mathbf{T}_{\mathfrak{g}(V)}}(W, W')@>>> \Hom_{\mathbf{T}_{\mathfrak{g}(V)}}(W, M') @>>> \Hom_{\mathbf{T}_{\mathfrak{g}(V)}}(W, N')\\
	@. @VVV @V{d}VV  @V{\cong}VV\\
	0 @>>>\Hom_{\mathbb{T}_{\mathfrak{g}(\infty)}}(W^{wt}, (W')^{wt})@>>> \Hom_{\mathbb{T}_{\mathfrak{g}(\infty)}}(W^{wt}, (M')^{wt}) @>>> \Hom_{\mathbb{T}_{\mathfrak{g}(\infty)}}(W^{wt}, (N')^{wt})
	\end{CD}
    \]
	where the rows are exact and the middle vertical arrow is the map $d$ from above. We have established that $d$ is an isomorphism, and in the same way one can establish that the third vertical arrow is an isomorphism. It thus follows that the first vertical arrow is also an isomorphism. This proves that the functor $(\phantom{a})^{wt}$ is fully faithful. 
	
	Let us prove the essential surjectivity of the functor $(\phantom{a})^{wt}$. Let $A$ be an object of $\mathbb{T}_{\mathfrak{g}(\infty)}$. By \cite[Proposition~4.5]{DPS}, $A$ may be assumed to be the kernel of a homomorphism $\phi\in \Hom_{\mathbb{T}_{\mathfrak{g}(\infty)}}((M'')^{wt}, (M''')^{wt})$, where $M'', M'''$ are finite direct sums of $\mathbf{V}^{p,q}$-s. By Lemma~\ref{lem:full}, the map $\phi$ is the restriction of a unique map $\psi\in \Hom_{\mathbf{T}_{\mathfrak{g}(V)}}(M'', M''')$. One has then $(\mathrm{ker}\psi)^{wt}=A$; thus the functor $(\phantom{a})^{wt}$ is essentially surjective. This concludes the proof of (\ref{tre}).
        \end{proof}

        \begin{cor} \label{quot}
	    Every object of $\widehat{\mathbf{T}}_{\mathfrak{g}(V)}$ is isomorphic to a quotient of a finite direct sum of $\widehat{\mathbf{V}}^{p,q}$-s.
	\end{cor}
        \begin{proof}
            Follows from Theorem~\ref{thm:main}(\ref{due}) and Proposition~\ref{prop:sub}.
        \end{proof}
	
	\section{A supersymmetric setting}

 We conclude the paper by a supersymmetric extension of Theorem~\ref{thm:main}. Let's start by introducing supersymmetric analogues of the categories $\mathbb{T}_{\mathfrak{gl}(\infty)}$, $\mathbf{T}_{\mathfrak{gl}(V)}$ and $\widehat{{\mathbf T}}_{\mathfrak{gl}(V)}$.

By $\Pi$ we denote the functor on the category of vector superspaces ($\mathbb{Z}/2\mathbb{Z}$-graded vector spaces) which changes the  $\mathbb{Z}/2\mathbb{Z}$-grading to the opposite one. Considering $V$ as a purely even superspace (i.e. $V_0 = V$, $V_1 = 0$), we introduce the topological superspace $V\oplus \Pi V$. This is a self-dual Tate superspace as $(V\oplus \Pi V)^\ast = V^\ast \oplus \Pi V^\ast = V \oplus \Pi V$. The Lie superalgebra of continuous endomorphisms of $V\oplus \Pi V$ is denoted $\mathfrak{gl}(V\oplus \Pi V)$. 
The Lie algebra $\mathfrak{gl}(V \oplus \Pi V)_0$, the even part of $\mathfrak{gl}(V \oplus \Pi V)$, equals $\mathfrak{gl}(V)\oplus \mathfrak{gl}(\Pi V)$. Note that $\mathfrak{gl}(\Pi V) = \mathfrak{gl}(V)$, i.e., $\mathfrak{gl}(V \oplus \Pi V)_0 = \mathfrak{gl}(V)\oplus \mathfrak{gl}(V)$.

 Let now $U \oplus \Pi U$ be the superspace with $(U \oplus \Pi U)_0 = U$ and $(U \oplus \Pi U)_1 = U$. Choosing a basis $\{u^0 _i\}$ in $U$ endows $U\oplus\Pi U$ with a basis $\{u^0 _i, u^1 _i\}$, where $u^1 _i$ is the vector $u^0 _i$ considered as an element of $\Pi U$.
 Setting then $U_\ast \oplus \Pi U_\ast = \mathrm{span}\{ (u_i ^0)^\ast ,  (u_i ^1)^\ast  \}$, where $\{(u_i ^0)^\ast\}$ is the system dual to $\{u^0 _i\}$ and $\{(u_i ^1)^\ast\}$  is the system dual to $\{u^1 _i\}$, we obtain a nondegenerate pairing 
 \[
 (U \oplus \Pi U) \times (U_\ast \oplus \Pi U_\ast) \rightarrow \mathbb{C}
 \]
 which respects the $\mathbb{Z}/2\mathbb{Z}$-grading.
 The tensor product 
 \[
 (U \oplus \Pi U) \otimes (U_\ast \oplus \Pi U_\ast)
 \]
is therefore an associative superalgebra ($\mathbb{Z}/2\mathbb{Z}$-graded associative algebra), and $\mathfrak{gl}(\infty | \infty)$ is by definition the Lie superalgebra associated with this superalgebra.	
Note that $\mathfrak{gl}(\infty | \infty)$ is a Lie subsuperalgebra of $\mathfrak{gl}(V \oplus \Pi V)$.

Similarly, $V\oplus \Pi V$ is endowed with a supersymmetric pairing
\[
AB: (V \oplus \Pi V) \times (V \oplus \Pi V) \rightarrow \mathbb{C}
\]
given by $AB((x,y), (z,t)) = A((x,z)) + B((y,t))$,
where $x,z\in V$, $y,t\in \Pi V$ and we identify $V$ and $\Pi V$ when taking $B((y,t))$.
The Lie subsuperalgebra
\[
\mathfrak{osp}(V \oplus \Pi V) \subset \mathfrak{gl}(V \oplus \Pi V)
\]
is the Lie subsuperalgebra generated by $\mathbb{Z}/2\mathbb{Z}$-homogeneous operators $\varphi \in \mathfrak{gl}(V \oplus \Pi V)$ of parity $\overline{\varphi}\in \mathbb{Z}/2\mathbb{Z}$, such that
\[
AB(\varphi((x,y)), (z,t)) + (-1)^{\overline{\varphi}}
AB((x,y), \varphi((z,t))) = 0 \ .
\]
The Lie algebra $\mathfrak{osp}(V \oplus \Pi V)_0$ is isomorphic to 
$\mathfrak{o}(V)\oplus \mathfrak{sp}(V)$.

The definitions of the categories $\mathbb{T}_{\mathfrak{gl}(\infty)}$, $\mathbf{T}_{\mathfrak{gl}(V)}$ and $\widehat{\mathbf{T}}_{\mathfrak{gl}(V)}$ extend in an obvious way to definitions of categories $\mathbb{T}_{\mathfrak{gl}(\infty | \infty)}$, $\mathbf{T}_{\mathfrak{gl}(V \oplus \Pi V)}$ and $\widehat{\mathbf{T}}_{\mathfrak{gl}(V\oplus \Pi V)}$: one replaces $V$ by $V\oplus \Pi V$
in the respective definition and works with vector superspaces instead of just vector spaces.

The following diagrams of functors emerge,
\begin{equation}
\label{eq1}
\begin{tikzcd}
    & \mathbf{T}_{\mathfrak{gl}(V \oplus \Pi V)} \arrow[dl, "i_{\Pi V}"'] \arrow[dr, "i_V"]& \\
    \mathbf{T}_{\mathfrak{gl}(V)} & & \mathbf{T}_{\mathfrak{gl}(\Pi V)}
\end{tikzcd}  
\end{equation}
and
\begin{equation}
\label{eq2}
\begin{tikzcd}
    & \widehat{\mathbf{T}}_{\mathfrak{gl}(V\oplus \Pi V)} \arrow[dl, "\widehat{i}_{\Pi V}"'] \arrow[dr, "\widehat{i}_V"]& \\
    \widehat{\mathbf{T}}_{\mathfrak{gl}(V)}& & \widehat{\mathbf{T}}_{\mathfrak{gl}(\Pi V)}
\end{tikzcd}  
\end{equation}
where $i_V$ and $\widehat{i}_V$ (respectively, $i_{\Pi V}$ and $\widehat{i}_{\Pi V}$) are the functors of taking $\mathfrak{gl}(V)$-invariants (respectively, $\mathfrak{gl}(\Pi V)$-invariants).

The weight functor $(\phantom{a})^{wt}$ transfers diagram (\ref{eq1}) to the diagram
\begin{equation}
\label{eq3}
\begin{tikzcd}
    & \mathbb{T}_{\mathfrak{gl}(\infty | \infty)} \arrow[dl, "i_{\Pi V, \infty}"'] \arrow[dr, "i_{V, \infty}"]& \\
    \mathbb{T}_{\mathfrak{gl}(\infty)} & & \mathbb{T}_{\mathfrak{gl}(\infty)} \ ,
\end{tikzcd}  
\end{equation}
where $i_{V, \infty}$ and $i_{\Pi V, \infty}$ are the respective functors of invariants with respect to the subalgebras 
$\mathfrak{gl}(\infty)\subset \mathfrak{gl}(V)$ and $\mathfrak{gl}(\infty)\subset \mathfrak{gl}(\Pi V)$.
Theorem~\ref{thm:main} asserts that the weight functor $(\phantom{a})^{wt}$ is an equivalence of the categories $\mathbf{T}_{\mathfrak{gl}(V)}$
and $\mathbb{T}_{\mathfrak{gl}(\infty)}$; we leave to the reader to check that this result extends also to the categories
$\mathbf{T}_{\mathfrak{gl}(V \oplus \Pi V)}$ and $\mathbb{T}_{\mathfrak{gl}(\infty | \infty)}$. Moreover, Serganova proves in \cite{S} that the diagram~(\ref{eq3}) consists of equivalences of monoidal categories, and this implies the existence of an equivalence of monoidal categories
\[
i_{V, \infty} \circ i_{\Pi V, \infty}^{-1}: \mathbb{T}_{\mathfrak{gl}(\infty)} \rightarrow \mathbb{T}_{\mathfrak{gl}(\infty)} \ .
\]

In this way, we arrive to
\begin{thm}\label{thm:main2}
    The diagrams (\ref{eq1}) and (\ref{eq2}) are diagrams of equivalences of monoidal categories.
\end{thm}
\begin{proof}
    The result for diagram~(\ref{eq1}) follows from Serganova's result and from the fact that the weight functor commutes with the respective functors of invariants. For the diagram (\ref{eq2}) the result follows by duality.
\end{proof}

Next we would like to turn our attention to the Lie superalgebra $\mathfrak{osp}(V \oplus \Pi V)$. We first recall the Lie superalgebra $\mathfrak{osp}(\infty | \infty)$. This is the intersection of the Lie subsuperalgebras 
$\mathfrak{osp}(V \oplus \Pi V)$ and $(V \oplus \Pi V)\otimes (V \oplus \Pi V)$ within $\mathfrak{gl}(V \oplus \Pi V)$,
where $V\oplus \Pi V$ is endowed with the pairing $AB$. The adjoint $\mathfrak{osp}(\infty | \infty)$-module is isomorphic to the second (super)exterior power $\bigwedge^2 (V \oplus \Pi V)$ of the superspace $V\oplus \Pi V$, and $\mathfrak{osp}(\infty | \infty)_0 \cong \mathfrak{o}(\infty)\oplus \mathfrak{sp}(\infty)$,
where $\mathfrak{o}(\infty) = \varinjlim \mathfrak{o}(n)$, $\mathfrak{sp}(\infty) = \varinjlim \mathfrak{sp}(2n)$.
The category $\mathbb{T}_{\mathfrak{osp}(\infty | \infty)}$ has been introduced in \cite{S} and its objects are 
$\mathfrak{osp}(\infty | \infty)$-modules isomorphic to subquotients of finite direct sums of tensor products of the form $(V \oplus \Pi V)^{\otimes k}$ for $k \geq 0$.

Next, the objects of the category $\mathbf{T}_{\mathfrak{osp}(V \oplus \Pi V)}$ are
topological vector superspaces of the form $Z/Q$, where $Z$ is a closed $\mathfrak{osp}(V \oplus \Pi V)$-stable subspace of a direct sum of topological vector superspaces of the form $(V\oplus \Pi V)^{\widehat{\otimes}^\ast p}$ and $Q$ is a closed $\mathfrak{osp}(V \oplus \Pi V)$-stable subspace of $Z$. The morphisms in the category $\mathbf{T}_{\mathfrak{osp}(V \oplus \Pi V)}$ are $\mathfrak{osp}(V \oplus \Pi V)$-equivariant continuous linear maps. Similarly to Theorem~\ref{thm:main} one shows that the weight part functor yields an equivalence between between the monoidal categories $\mathbf{T}_{\mathfrak{osp}(V \oplus \Pi V)}$ and $\mathbb{T}_{\mathfrak{osp}(\infty | \infty)}$.

Moreover, Serganova has established in \cite{S} that the diagram 
\begin{equation}
\label{eq4}
\begin{tikzcd}
    & \mathbb{T}_{\mathfrak{osp}(\infty | \infty)} \arrow[dl, "i_{\mathfrak{sp}(\infty)}"'] \arrow[dr, "i_{\mathfrak{o}(\infty)}"]& \\
    \mathbb{T}_{\mathfrak{o}(\infty)} & & \mathbb{T}_{\mathfrak{sp}(\infty)} \ ,
\end{tikzcd}  
\end{equation}
where $i_{\mathfrak{o}(\infty)}$ and $i_{\mathfrak{sp}(\infty)}$ are the respective functors of invariants with respect to the subalgebras $\mathfrak{o}(\infty) \subset \mathfrak{o}(V)$ and $\mathfrak{sp}(\infty) \subset \mathfrak{sp}(V)$, is a diagram of equivalences of monoidal categories.

This brings us to
\begin{thm}
    The diagrams
    \begin{equation}
\label{eq5}
\begin{tikzcd}
    & \mathbf{T}_{\mathfrak{osp}(V \oplus \Pi V)} \arrow[dl, "i_{\mathfrak{sp}(V)}"'] \arrow[dr, "i_{\mathfrak{o}(V)}"]& \\
    \mathbf{T}_{\mathfrak{o}(V)} & & \mathbf{T}_{\mathfrak{sp}(V)}
\end{tikzcd}  
\end{equation}
and
\begin{equation}
\label{eq6}
\begin{tikzcd}
    & \widehat{\mathbf{T}}_{\mathfrak{osp}(V \oplus \Pi V)} \arrow[dl, "\widehat{i}_{\mathfrak{sp}(V)}"'] \arrow[dr, "\widehat{i}_{\mathfrak{o}(V)}"]& \\
    \widehat{\mathbf{T}}_{\mathfrak{o}(V)}& & \widehat{\mathbf{T}}_{\mathfrak{sp}(V)} \ ,
\end{tikzcd}  
\end{equation}
where $i_{\mathfrak{o}(V)}$ and $\ \widehat{i}_{\mathfrak{o}(V)}$ (respectively, $i_{\mathfrak{sp}(V)}$ and $\ \widehat{i}_{\mathfrak{sp}(V)}$) are the functors of taking $\mathfrak{o}(V)$-invariants (respectively, $\mathfrak{sp}(V)$-invariants), are diagrams of equivalences of monoidal categories.
\end{thm}
\begin{flushright}
$\Box$
\end{flushright}

We complete this paper by a brief discussion of universality properties of all our categories. Recall first, that the monoidal category $\mathbb{T}_{\mathfrak{gl}(\infty)}$ admits a canonical left exact functor into any $\mathbb{C}$-linear symmetric abelian monoidal category with two objects $X,Y$ endowed with a morphism $X\otimes Y \rightarrow \mathbf{1}$, where $\mathbf{1}$ denotes the monoidal unit. This functor sends $U$ to $X$, $U_\ast$ to $Y$, and the pairing $U\otimes U_\ast \rightarrow \mathbb{C}$ to the morphism $X\otimes Y \rightarrow \mathbf{1}$. The existence of such a functor has been established in \cite{SS} and \cite{DPS}. Therefore, any category which is equivalent to $\mathbb{T}_{\mathfrak{gl}(\infty)}$ as a monoidal category has the same universality property. According to Theorems~\ref{thm:main} and \ref{thm:main2}, this applies to the categories 
$\mathbf{T}_{\mathfrak{gl}(U)}$, $\mathbf{T}_{\mathfrak{gl}(V)}$, $\mathbf{T}_{\mathfrak{gl}(U\oplus \Pi U)}$ and $\mathbf{T}_{\mathfrak{gl}(V\oplus \Pi V)}$.

Since the category $\widehat{\mathbf{T}}_{\mathfrak{gl}(U)}$ is anti-equivalent to the category $\mathbf{T}_{\mathfrak{gl}(U)}$ \cite{EP}, there is a canonical right-exact functor from $\widehat{\mathbf{T}}_{\mathfrak{gl}(U)}$ to any $\mathbb{C}$-linear symmetric monoidal category with two objects $X,Y$ and a morphism $\mathbf{1}\rightarrow X\otimes Y$.
This functor sends $U$ to $X$, $U^\ast$ to $Y$, and the injection $\mathbf{1}\rightarrow U\widehat{\otimes}^! U^\ast$ to the morphism $\mathbf{1}\rightarrow X\otimes Y$.
Therefore this universality property applies to both monoidal categories $\widehat{\mathbf{T}}_{\mathfrak{gl}(V)}$ and 
$\widehat{\mathbf{T}}_{\mathfrak{gl}(V\oplus \Pi V)}$ introduced in this paper.

Finally, we recall that the equivalent monoidal categories $\mathbb{T}_{\mathfrak{o}(\infty)}$ and $\mathbb{T}_{\mathfrak{sp}(\infty)}$ share the following universality property: each of them admits a canonical left-exact functor to any given $\mathbb{C}$-linear symmetric monoidal category with an object $X$ endowed with a morphism $X\otimes X \rightarrow \mathbf{1}$,
sending $V$, or respectively $V\oplus \Pi V$, to $X$. Since the categories $\mathbf{T}_{\mathfrak{o}(V)}$, $\mathbf{T}_{\mathfrak{sp}(V)}$ and $\mathbf{T}_{\mathfrak{osp}(V\oplus \Pi V)}$ are equivalent as monoidal categories to $\mathbb{T}_{\mathfrak{o}(\infty)}$ and $\mathbb{T}_{\mathfrak{sp}(\infty)}$, they share this universality property. 
In turn, the categories $\widehat{\mathbf{T}}_{\mathfrak{o}(V)}$. $\widehat{\mathbf{T}}_{\mathfrak{sp}(V)}$ and $\widehat{\mathbf{T}}_{\mathfrak{osp}(V \oplus \Pi V)}$ share the following universality property: each of them admits a canonical right-exact functor into any given $\mathbb{C}$-linear symmetric monoidal category with an object $X$ endowed with a morphism $ \mathbf{1} \rightarrow X\otimes X $, sending $V$, or respectively $V\oplus \Pi V$, to $X$.

\end{document}